\documentclass[11pt]{amsart}
\usepackage{amsopn}
\usepackage{amssymb, amscd}
\usepackage{multirow}

\newcommand{\nc}{\newcommand}

\nc{\fg}{\mathfrak{f} } \nc{\vg}{\mathfrak{v} } \nc{\wg}{\mathfrak{w} }
\nc{\zg}{\mathfrak{z} } \nc{\ngo}{\mathfrak{n} } \nc{\kg}{\mathfrak{k} }
\nc{\mg}{\mathfrak{m} } \nc{\bg}{\mathfrak{b} } \nc{\ggo}{\mathfrak{g} }
\nc{\ggob}{\overline{\mathfrak{g}} } \nc{\sog}{\mathfrak{so} }
\nc{\sug}{\mathfrak{su} } \nc{\spg}{\mathfrak{sp} } \nc{\slg}{\mathfrak{sl} }
\nc{\glg}{\mathfrak{gl} } \nc{\cg}{\mathfrak{c} } \nc{\rg}{\mathfrak{r} }
\nc{\hg}{\mathfrak{h} } \nc{\tg}{\mathfrak{t} } \nc{\ug}{\mathfrak{u} }
\nc{\dg}{\mathfrak{d} } \nc{\ag}{\mathfrak{a} } \nc{\pg}{\mathfrak{p} }
\nc{\sg}{\mathfrak{s} } \nc{\affg}{\mathfrak{aff} } \nc{\qg}{\mathfrak{q} }

\nc{\pca}{\mathcal{P}} \nc{\nca}{\mathcal{N}} \nc{\lca}{\mathcal{L}}
\nc{\oca}{\mathcal{O}} \nc{\mca}{\mathcal{M}} \nc{\tca}{\mathcal{T}}
\nc{\aca}{\mathcal{A}} \nc{\cca}{\mathcal{C}} \nc{\gca}{\mathcal{G}}
\nc{\sca}{\mathcal{S}} \nc{\hca}{\mathcal{H}} \nc{\bca}{\mathcal{B}}
\nc{\dca}{\mathcal{D}} \nc{\val}{\operatorname{val}}

\nc{\vp}{\varphi} \nc{\ddt}{\frac{d}{dt}} \nc{\dds}{\frac{d}{ds}}
\nc{\dpar}{\frac{\partial}{\partial t}} \nc{\im}{\mathrm{i}}

\nc{\SO}{\mathrm{SO}} \nc{\Spe}{\mathrm{Sp}} \nc{\Sl}{\mathrm{SL}}
\nc{\SU}{\mathrm{SU}} \nc{\Or}{\mathrm{O}} \nc{\U}{\mathrm{U}} \nc{\Gl}{\mathrm{GL}}
\nc{\Se}{\mathrm{S}} \nc{\Cl}{\mathrm{Cl}} \nc{\Spein}{\mathrm{Spin}}
\nc{\Pin}{\mathrm{Pin}} \nc{\G}{\mathrm{GL}_n(\RR)} \nc{\g}{\mathfrak{gl}_n(\RR)}

\nc{\RR}{{\Bbb R}} \nc{\HH}{{\Bbb H}} \nc{\CC}{{\Bbb C}} \nc{\ZZ}{{\Bbb Z}}
\nc{\FF}{{\Bbb F}} \nc{\NN}{{\Bbb N}} \nc{\QQ}{{\Bbb Q}} \nc{\PP}{{\Bbb P}} \nc{\OO}{{\Bbb O}}

\nc{\vs}{\vspace{.2cm}} \nc{\vsp}{\vspace{1cm}} \nc{\ip}{\langle\cdot,\cdot\rangle}
\nc{\ipp}{(\cdot,\cdot)} \nc{\la}{\langle} \nc{\ra}{\rangle} \nc{\unm}{\tfrac{1}{2}}
\nc{\unc}{\tfrac{1}{4}} \nc{\und}{\tfrac{1}{16}} \nc{\no}{\vs\noindent}
\nc{\lam}{\Lambda^2(\RR^n)^*\otimes\RR^n} \nc{\tangz}{{\rm T}^{\rm Zar}}
\nc{\nor}{{\sf n}}  \nc{\mum}{/\!\!/} \nc{\kir}{/\!\!/\!\!/}
\nc{\Ri}{\tfrac{4\Ric_{\mu}}{||\mu||^2}} \nc{\ds}{\displaystyle}
\nc{\ben}{\begin{enumerate}} \nc{\een}{\end{enumerate}} \nc{\f}{\frac}
\nc{\lb}{[\cdot,\cdot]} \nc{\isn}{\tfrac{1}{||v||^2}}
\nc{\gkp}{(\ggo=\kg\oplus\pg,\ip)} \nc{\ukh}{(\ug=\kg\oplus\hg,\ip)}
\nc{\tgkp}{(\tilde{\ggo}=\kg\oplus\pg,\ip)}
\nc{\wt}{\widetilde} \nc{\mm}{M}
\nc{\iop}{\mathtt{i}} \nc{\jop}{\mathtt{j}}

\nc{\Hess}{\operatorname{Hess}} \nc{\ad}{\operatorname{ad}}
\nc{\Ad}{\operatorname{Ad}} \nc{\rank}{\operatorname{rank}}
\nc{\Irr}{\operatorname{Irr}} \nc{\End}{\operatorname{End}}
\nc{\Aut}{\operatorname{Aut}} \nc{\Inn}{\operatorname{Inn}}
\nc{\Der}{\operatorname{Der}} \nc{\Ker}{\operatorname{Ker}}
\nc{\Iso}{\operatorname{Iso}} \nc{\Diff}{\operatorname{Diff}}
\nc{\Lie}{\operatorname{L}} \nc{\tr}{\operatorname{tr}} \nc{\dif}{\operatorname{d}}
\nc{\sen}{\operatorname{sen}} \nc{\modu}{\operatorname{mod}}
\nc{\CRic}{\operatorname{PP}} \nc{\Cric}{\operatorname{P}} \nc{\Ricci}{\operatorname{Ric}}
\nc{\sym}{\operatorname{sym}} \nc{\herm}{\operatorname{herm}} \nc{\symac}{\operatorname{sym^{ac}}}
\nc{\symc}{\operatorname{sym^{c}}} \nc{\scalar}{\operatorname{sc}}
\nc{\grad}{\operatorname{grad}} \nc{\ricci}{\operatorname{Rc}}
\nc{\Nor}{\operatorname{Norm}}  \nc{\ricc}{\operatorname{Rc^{c}}}
\nc{\Ricc}{\operatorname{Ric^{c}}} \nc{\ricac}{\operatorname{Rc^{ac}}}
\nc{\Ricac}{\operatorname{Ric^{ac}}} \nc{\Riem}{\operatorname{Rm}}
\nc{\riccig}{\operatorname{ric^{\gamma}}} \nc{\Rin}{\operatorname{M}}
\nc{\Le}{\operatorname{L}} \nc{\tang}{\operatorname{T}}
\nc{\level}{\operatorname{level}} \nc{\rad}{\operatorname{r}}
\nc{\abel}{\operatorname{ab}} \nc{\CH}{\operatorname{CH}}
\nc{\mcc}{\operatorname{mcc}} \nc{\Adj}{\operatorname{Adj}}
\nc{\Order}{\operatorname{O}}  \nc{\inj}{\operatorname{inj}} \nc{\proy}{\operatorname{pr}}
\nc{\vol}{\operatorname{vol}} \nc{\Diag}{\operatorname{Dg}}
\nc{\Spec}{\operatorname{Spec}} \nc{\Ima}{\operatorname{Im}} \nc{\Rea}{\operatorname{Re}}
\nc{\spann}{\operatorname{span}}

\theoremstyle{plain}
\newtheorem{theorem}{Theorem}[section]
\newtheorem{proposition}[theorem]{Proposition}

\newtheorem{lemma}[theorem]{Lemma}

\theoremstyle{definition}
\newtheorem{definition}[theorem]{Definition}

\theoremstyle{remark}
\newtheorem{remark}[theorem]{Remark}

\newtheorem{example}[theorem]{Example}
\newtheorem{question}[theorem]{Question}

\title[Laplacian solitons]{Laplacian solitons: questions and homogeneous examples}

\author{Jorge Lauret}

\address{Universidad Nacional de C\'ordoba, FaMAF and CIEM, 5000 C\'ordoba, Argentina}
\email{lauret@famaf.unc.edu.ar}

\thanks{This research was partially supported by grants from CONICET, FONCYT and SeCyT (Universidad Nacional de C\'ordoba)}

\begin{document}

\maketitle

\begin{abstract}
We give the first examples of closed Laplacian solitons which are shrinking, and in particular produce closed Laplacian flow solutions with a finite-time singularity.  Extremally Ricci pinched $G_2$-structures (introduced by Bryant) which are steady Laplacian solitons have also been found.  All the examples are left-invariant $G_2$-structures on solvable Lie groups.
\end{abstract}


\section{Introduction}\label{intro}

A $G_2$-{\it structure} on a $7$-dimensional differentiable manifold $M$ is a positive (or definite) differential $3$-form on $M$.  Each $G_2$-structure $\vp$ defines a Riemannian metric $g$ on $M$ and an orientation.  In the case when $\vp$ is closed, which will be assumed throughout this paper, the only torsion form that survives is a $2$-form $\tau$, and the starting situation can be described as follows:
\begin{equation}\label{tauphi-intro}
d\vp=0, \qquad \tau=-\ast d\ast\vp, \qquad d\ast\vp=\tau\wedge\vp, \qquad d\tau=\Delta\vp,
\end{equation}
where $\Delta$ denotes the Hodge Laplacian operator determined by $g$.

The following nice interplay between the metric and the torsion $2$-form of a closed $G_2$-structure was discovered by R. Bryant.  Let $R$ and $\Ricci$ denote the scalar and Ricci curvature of the associated metric $g$, respectively.

\begin{theorem}\cite[Corollary 3]{Bry}\label{Bineq-intro}
If $\vp$ is a closed $G_2$-structure on a compact manifold $M$, then
$$
\int_M R^2 \ast 1 \leq 3\int_M |\Ricci|^2 \ast 1,
$$
and equality holds if and only if $\; d\tau = \frac{1}{6}|\tau|^2\vp + \frac{1}{6}\ast(\tau\wedge\tau)$.
\end{theorem}

The special $G_2$-structures for which equality holds were called {\it extremally Ricci-pinched} (ERP for short) in \cite[Remark 13]{Bry}.   Note that the factor of $3$ on the right hand side is much smaller than the a priori factor of $7$ provided by the Cauchy-Schwartz inequality, which is attained precisely at Einstein metrics.

Another distinguished class of closed $G_2$-structures are {\it Laplacian solitons} (LS for short):
$$
\Delta\vp=c\vp+\lca_{X}\vp, \qquad \mbox{for some}\quad c\in\RR, \quad X\in\mathfrak{X}(M)\; \mbox{(complete)},
$$
which are also characterized as those $G_2$-structures flowing in a self-similar way along the Laplacian flow $\dpar\vp(t) = \Delta\vp(t)$.  They are respectively called {\it expanding}, {\it steady} or {\it shrinking}, if $c>0$, $c=0$ or $c<0$.

In \cite{LF}, the Laplacian flow and its solitons were studied on homogeneous spaces.  The present paper consists of a collection of noncompact homogeneous examples that were obtained as applications of some of the structural results in \cite{LF}.  We consider the invariant functional
$$
F:=\frac{R^2}{|\Ricci|^2}, \qquad 0\leq F\leq 7,
$$
on the space of all non-flat homogeneous $G_2$-structures.  The following natural questions arise from a comparison between the compact and homogeneous cases:

\begin{quote}
(Q1) Does the inequality $F\leq 3$ hold for any homogeneous closed $G_2$-structure?
\end{quote}

Note that this holds after integrating on $M$ in the compact case by Theorem \ref{Bineq-intro}.  We expect that $F<7$, considering that no solvable Lie group admits a left-invariant Einstein closed (non-parallel) $G_2$-structure (see \cite{FrnFinMnr2}).

\begin{quote}
(Q2) Are there shrinking or steady (non-parallel) homogeneous LSs?
\end{quote}

In the compact case, it was proved in \cite{Lin} that there are no shrinking Laplacian solitons and that the only steady ones are parallel (or torsion-free, i.e.\ $\tau=0$) $G_2$-structures.  On the other hand, examples of closed expanding LSs were given on solvable and nilpotent Lie groups in \cite{BF, LF, Ncl}.

Within the class of left-invariant closed $G_2$-structures on solvable Lie groups with a codimension one abelian normal subgroup, the above questions were answered in \cite[Section 5]{LF}: all Laplacian flow solutions are immortal, any LS is either expanding or parallel and $F\leq 1$.  Moreover, equality $F=1$ holds  precisely at non-nilpotent expanding LSs.  The following is the main result of this paper, showing that none of these properties is longer true for general solvable Lie groups (see Examples \ref{Htype} and \ref{Htype-ineq}).

\begin{theorem}
There is a one-parameter family of pairwise non-isomorphic solvable Lie groups $G_a$, $a\in\left[\unc,\infty\right)$, each endowed with a left-invariant closed $G_2$-structure, such that:

\begin{itemize}
\item $G_a$ is a shrinking LS for any $a\in\left[\unc,1\right)$.  In particular, the Laplacian flow solution $\vp(t)$ starting at any of these $G_2$-structures develops a singularity at a finite time $T$.  Moreover, $\vp(t)$ converges to zero as $t\to T$ at a rate of $(T-t)^{\alpha}$ with $\unc\leq\alpha<\unm$ depending on $a$.

\item $G_1$ is ERP and a steady LS.  There is a compact quotient $M$ of $G_1$ locally equivalent to $G_1$; in particular, $M$ is also ERP and the corresponding Laplacian flow solution on $M$ is eternal and `locally' self-similar.

\item $G_a$ is an expanding LS for any $a\in(1,\infty)$.

\item $F$ is strictly decreasing on $\left[\unc,\infty\right)$ and
$$
F(\unc)=\tfrac{81}{17} \; (\approx 4.76) \quad > \quad  3=F(1) \quad > \quad 1=\lim_{a\to\infty} F(a).
$$
\end{itemize}
\end{theorem}

The way these examples were actually found is through a general construction of homogeneous $G_2$-structures as extensions of $\SU(3)$-structures, which will be developed in detail in a forthcoming paper (see \cite{Mnr,FinRff1,FinRff2,FrnFinRff} for other particular cases of such construction).  Some additional geometric information on this family is that $G_a$ is a Ricci soliton if and only if $a=1$, $G_a$ has $\Ricci\leq 0$ if and only if $a\in\left[\unc,1\right]$ and has always mixed sectional curvature.  We do not know what is the supremum of $F$ in the homogeneous case, nor how special its maximal points are (if they exist at all).

Furthermore, we obtain the following results and examples:

\begin{itemize}
\item The homogeneous example of an ERP $G_2$-structure given in \cite[Example 1]{Bry} is a steady LS, equivalent to $G_1$ above (see Examples \ref{Bryant2} and \ref{Bryant3}).

\item A new example of a left-invariant $G_2$-structure on a solvable Lie group which is ERP and also a steady LS (see Example \ref{tripleA}).

\item There is no positive lower bound for $F$ and $F$ diverges at a flat $G_2$-structure (see Examples \ref{muAex} and \ref{muAex2}).
\end{itemize}

\vs \noindent {\it Acknowledgements.} This material is based upon work supported by the NSF under Grant No.\ DMS-1440140 while the author was in residence at the MSRI in Berkeley, California, during the Spring 2016 semester.

\section{Preliminaries}\label{prelim}

We give in this section a short overview on $G_2$-structures and some of the linear algebra involved, see \cite{Bry,Lty} for much more complete treatments.

\subsection{$G_2$-structures}
A $G_2$-{\it structure} on a $7$-dimensional differentiable manifold $M$ is a differential $3$-form $\vp\in\Omega^3M$ such that $\vp_p$ is {\it positive} on $T_pM$ for any $p\in M$, that is, $\vp_p$ can be written as
\begin{equation}\label{phican}
\vp_p=\omega\wedge e^7+\rho^+=e^{127}+e^{347}+e^{567}+e^{135}-e^{146}-e^{236}-e^{245},
\end{equation}
where
$$
\omega:=e^{12}+e^{34}+e^{56}, \qquad \rho^+:=e^{135}-e^{146}-e^{236}-e^{245},
$$
with respect to some basis $\{ e_1,\dots,e_7\}$ of $T_pM$.  The usual notation $e^{ij\cdots}$ to indicate $e^i\wedge e^j\wedge\cdots$ will be freely used throughout the paper.  Each $G_2$-structure $\vp$ defines a Riemannian metric $g$ on $M$ and an orientation $\vol\in\Omega^7M$ (unique up to scaling) such that
$$
g(X,Y)\vol = \frac{1}{6} i_X(\vp)\wedge i_Y(\vp)\wedge\vp, \qquad\forall X,Y\in\mathfrak{X}(M),
$$
and we denote by $\ast:\Omega M\longrightarrow\Omega M$ the corresponding Hodge star operator.  Thus $\vp$ also determines the {\it Hodge Laplacian operator}
$\Delta:\Omega^kM\longrightarrow\Omega^kM$, $\Delta:=d^*d+dd^*$, where $d^*:\Omega^{k+1}M\longrightarrow\Omega^kM$, $d^*=(-1)^{k+1}\ast d\ast$, is the adjoint of $d$.  Let $\ricci$ denote the Ricci tensor of $(M,g)$.

The {\it torsion forms} of a $G_2$-structure $\vp$ on $M$ are the components of the {\it intrinsic torsion} $\nabla\vp$, where $\nabla$ is the Levi-Civita connection of the metric $g$.  They can be defined as the unique differential forms $\tau_i\in\Omega^iM$, $i=0,1,2,3$, such that
\begin{equation}\label{dphi}
d\vp=\tau_0\ast\vp+3\tau_1\wedge\vp+\ast\tau_3, \qquad d\ast\vp=4\tau_1\wedge\ast\vp+\tau_2\wedge\vp.
\end{equation}
Some special classes of $G_2$-structures are defined or characterized as follows:

\begin{itemize}
\item {\it parallel} or {\it torsion-free}: $d\vp=0$ and $d\ast\vp=0$, or equivalently, $\nabla\vp=0$ (for $M$ compact, this is equivalent to $\vp$ {\it harmonic}: $\Delta\vp=0$);
\item {\it closed} or {\it calibrated}: $d\vp=0$;
\item {\it eigenform}: $\Delta\vp=c\vp$ for some $c\in\RR$;
\item {\it Einstein}: $\ricci=cg$ for some $c\in\RR$;
\item {\it Laplacian soliton} (LS for short): $\Delta\vp = c\vp + \lca_X\vp$ for some $c\in\RR$ and $X\in\mathfrak{X}(M)$ (called {\it expanding}, {\it steady} or {\it shrinking} if $c>0$, $c=0$ or $c<0$, respectively);
\item {\it Ricci soliton}: $\ricci = cg + \lca_Xg$ for some $c\in\RR$ and $X\in\mathfrak{X}(M)$.
\end{itemize}

The {\it Laplacian flow} for a family $\vp(t)$ of $G_2$-structures on $M$ is the evolution equation
$$
\dpar\vp(t) = \Delta\vp(t),
$$
introduced by Bryant in \cite[Section 6]{Bry} (see also \cite{BryXu,Krg,Grg,Lty,Lty2,Lty3}).  It follows from the invariance by diffeomorphisms of the flow that a solution $\vp(t)$ starting at a $G_2$-structure $\vp$ will be {\it self-similar}, in the sense that $\vp(t)=c(t)f(t)^*\vp$, for some $c(t)\in\RR^*$ and $f(t)\in\Diff(M)$, if and only if $\vp$ is a Laplacian soliton.  Analogously, {\it Ricci solitons} correspond to self-similar solutions to the well-known {\it Ricci flow} $\dpar g=-2\ricci(g)$.

\subsection{Some linear algebra}
We next show how to translate some of the $3$-form analysis into linear operators of the tangent space.  For a given point $p\in M$, let us fix a positive $3$-form $\vp$ on $\pg:=T_pM$, a real vector space of dimension $7$, and denote by $\ip$ the inner product on $\pg$ defined by $\vp$.  Since the orbit $\Gl(\pg)\cdot\vp$ is open in $\Lambda^3\pg^*$, we have that its tangent space at $\vp$ satisfies
\begin{equation}\label{nondeg}
\theta(\glg(\pg))\vp = \displaystyle{\Lambda^3}\pg^*,
\end{equation}
where $\theta:\glg(\pg)\longrightarrow \End(\Lambda^3\pg^*)$ is the representation obtained as the derivative of the natural left $\Gl(\pg)$-action on $3$-forms $h\cdot\psi=\psi(h^{-1}\cdot,h^{-1}\cdot,h^{-1}\cdot)$, that is,
$$
\theta(A)\psi=-\psi(A\cdot,\cdot,\cdot)-\psi(\cdot,A\cdot,\cdot)-\psi(\cdot,\cdot,A\cdot), \qquad\forall A\in\glg(\pg),\quad\psi\in\Lambda^3\pg^*.
$$
The Lie algebra of the stabilizer subgroup $G_2=\Gl(\pg)_\vp$ is given by
$$
\ggo_2:=\{ A\in\glg(\pg):\theta(A)\vp=0\}.
$$
We consider the following $G_2$-invariant decompositions, which are orthogonal relative to the inner product on $\glg(\pg)$ determined by $\ip$ (i.e.\ $\tr{AB^t}$):
\begin{equation}\label{g2-dec}
\begin{array}{c}
\glg(\pg)=\ggo_2\oplus\qg, \qquad \qg=\qg_1\oplus\qg_7\oplus\qg_{27}, \\ \\
\sog(\pg)=\ggo_2\oplus\qg_7, \qquad \sym(\pg)=\qg_1\oplus\qg_{27}, \qquad \qg_{27}=\sym_0(\pg),
\end{array}
\end{equation}
where $\sog(\pg)$ and $\sym(\pg)$ are the spaces of skew-symmetric and symmetric linear maps of $\pg$ with respect to $\ip$, respectively.  Furthermore,
$\sym_0(\pg):=\{ A\in\sym(\pg):\tr{A}=0\}$, $\qg_1=\RR I$ is the one-dimensional trivial representation of $G_2$, $\qg_7$ the standard representation and $\qg_{27}$ the other fundamental representation of $G_2$.  By setting $\Lambda^3_i\pg^*:=\theta(\qg_i)\vp$, for $i=1,7,27$, one obtains the decomposition
$$
\Lambda^3\pg^* = \Lambda^3_{1}\pg^*\oplus\Lambda^3_{7}\pg^*\oplus\Lambda^3_{27}\pg^*,
$$
of the space of $3$-forms in irreducible $G_2$-representations.

It follows that $\theta(\qg)\vp=\Lambda^3\pg^*$ (see \eqref{nondeg}); moreover, for every $3$-form $\psi\in\Lambda^3\pg^*$, there exists a unique operator $Q_\psi\in\qg$ such that
\begin{equation}\label{nondeg2}
\psi=\theta(Q_\psi)\vp.
\end{equation}
If we use $\ip$ to identify $\sym(\pg)$ with the space $S^2\pg^*$ of symmetric bilinear forms, then the linear isomorphism
$
\iop:S^2\pg^*\equiv\sym(\pg)\longrightarrow \Lambda^3_{1}\pg^*\oplus\Lambda^3_{27}\pg^*,
$
defined in \cite[(2.17)]{Bry} (and in \cite[(2.6)]{Lty} with a factor of $1/2$) is given by
\begin{equation}\label{iop}
\iop(A)=-2\theta(A)\vp; \qquad\mbox{in particular}, \quad \iop(Q_\psi)=-2\psi.
\end{equation}
On the other hand, the linear map $\jop:\Lambda^3\pg^*\longrightarrow\sym(\pg)$ defined in \cite[(2.18)]{Bry} (and in \cite{Lty}) satisfies that $\jop(\iop(h))=8h+4\tr(h)\ip$ for any $h\in S^2\pg^*$.  It follows from \eqref{iop} that, in terms of the $Q$-operators, $\jop$ is defined by
\begin{equation}\label{jop}
\jop(\psi)=-2\tr(Q_\psi)I-4Q_\psi, \qquad\forall\psi\in\Lambda^3_{1}\pg^*\oplus\Lambda^3_{27}\pg^*.
\end{equation}
Recall that $\jop$ vanishes on $\Lambda^3_{7}\pg^*$ and it is an isomorphism when restricted to $\Lambda^3_{1}\pg^*\oplus\Lambda^3_{27}\pg^*$.

\subsection{Homogeneous $G_2$-structures}\label{homog}
We refer to \cite{LF} for a more complete study of homogeneous $G_2$-structures.  A $7$-manifold endowed with a $G_2$-structure $(M,\vp)$ is said to be {\it homogeneous} if the Lie group of all its automorphisms (or symmetries),
$$
\Aut(M,\vp):=\{ f\in\Diff(M):f^*\vp=\vp\}\subset\Iso(M,g),
$$
acts transitively on $M$.  Each Lie subgroup $G\subset\Aut(M,\vp)$ which is transitive on $M$ gives rise to a presentation of $M$ as a {\it homogeneous space} $G/K$, where $K$ is the isotropy subgroup of $G$ at some point $o\in M$.  In that case, $\vp$ becomes a $G$-invariant $G_2$-structure on the homogeneous space $M=G/K$.  In the presence of a {\it reductive} decomposition $\ggo=\kg\oplus\pg$ (i.e.\ $\Ad(K)\pg\subset\pg$) for the homogeneous space $G/K$, there is a natural correspondence between the set of all $G$-invariant $G_2$-structures on $G/K$ and the set of all positive $3$-forms on $\pg\equiv T_oG/K$ which are $\Ad(K)$-invariant.  More in general, $\left(\Omega^k(G/K)\right)^G\equiv\left(\Lambda^k\pg^*\right)^K$.

A diffeomorphism $f$ of a homogeneous space $G/K$ is said to be {\it equivariant} if $f(hK)=\tilde{f}(h)K$ for all $h\in G$ for some $\tilde{f}\in\Aut(G)$ such that $\tilde{f}(K)=K$.  Let $\Aut(G/K)$ denote the subgroup of $\Diff(G/K)$ of all equivariant diffeomorphisms of $G/K$.  If $G$ is simply connected and $K$ connected, then each derivation $D\in\Der(\ggo)$ such that $D\kg\subset\kg$ defines a one-parameter subgroup $e^{tD}\in\Aut(G)\equiv\Aut(\ggo)$ taking $K$ onto $K$.  The vector field on $G/K$ defined by the corresponding one-parameter subgroup of $\Aut(G/K)$ will be denoted by $X_D$.  It is easy to see that for any $G$-invariant $k$-form $\psi$ on $G/K$, one has $\lca_{X_D}\psi=-\theta(D)\psi$.

\section{Quadratic $G_2$-structures}

In this section, we consider some distinguished classes of closed $G_2$-structures introduced and studied by Bryant in \cite{Bry}, which have notable properties in the compact case.  Our main interest here is in making a comparison with the (not necessarily compact) homogeneous case, in the search of parallelisms.  In this respect, a first natural question is

\begin{question}\label{comphom}
  Does there exist a homogeneous closed (non-parallel) $G_2$-structure on a compact manifold?
\end{question}

We note that the canonical $G_2$-structure on $\RR^7$ descends to a parallel $G_2$-structure on the torus $T^7$, which is clearly homogeneous.  The classification of compact homogeneous spaces admitting an invariant $G_2$-structure has been achieved in \cite{Rdg,VanMnr}, though the above question remains open.

Recall that a $G_2$-structure on a $7$-manifold $M$ is closed if and only if the torsion forms $\tau_0$, $\tau_1$ and $\tau_3$ all vanish (see \eqref{dphi}).  The only torsion form that survives is therefore the $2$-form $\tau_2$, which will be denoted by $\tau$ in the closed case.  It follows that
\begin{equation}\label{tauphi}
d\vp=0, \qquad \tau=-\ast d\ast\vp, \qquad d\ast\vp=\tau\wedge\vp, \qquad d\tau=\Delta\vp.
\end{equation}

The Ricci tensor and scalar curvature of a closed $G_2$-structure $\vp$ (see \cite[4.37]{Bry}) are respectively given by
\begin{equation}\label{ric}
\ricci = \unc|\tau|^2g - \unc\jop\left(d\tau-\unm\ast(\tau\wedge\tau)\right), \qquad R = -\unm|\tau|^2.
\end{equation}
In particular, a closed $G_2$-structure is parallel if and only if it is Ricci flat, if and only if $R=0$.  It follows from \eqref{jop} that the Ricci operator equals
\begin{equation}\label{ricop}
\Ricci = \left(\unc|\tau|^2+\unm\tr{Q_{d\tau}}-\unc\tr{Q_{\ast(\tau\wedge\tau)}}\right)I + Q_{d\tau} -\unm Q_{\ast(\tau\wedge\tau)},
\end{equation}
where $Q_{d\tau}$ and $Q_{\ast(\tau\wedge\tau)}$ are the symmetric operators defined as in \eqref{nondeg2} for the $3$-forms $d\tau$ and $\ast(\tau\wedge\tau)$, respectively.  Another formula for $\Ricci$ was given in \cite[Proposition 2.2]{LF}:
\begin{equation}\label{ricop2}
\Ricci = -\tfrac{1}{6}|\tau|^2I + Q_{d\tau} -\unm\tau^2,
\end{equation}
where $\tau\in\sog(\pg)$ also denotes the skew-symmetric linear map determined by the $2$-form $\tau$ (i.e.\ $\tau=\la\tau\cdot,\cdot\ra$).  This implies that $\tr{Q_{d\tau}}=-\tfrac{1}{3}|\tau|^2$ since $\tr{\tau^2}=-2|\tau|^2$, so by using \eqref{ricop} and \eqref{ricop2} we obtain that
\begin{equation}\label{att}
\tr{Q_{\ast(\tau\wedge\tau)}}=\tfrac{1}{3}|\tau|^2, \qquad \Ricci=Q_{d\tau}-\unm Q_{\ast(\tau\wedge\tau)},
\end{equation}
\begin{equation}\label{att2}
Q_{\ast(\tau\wedge\tau)}=\tfrac{1}{3}|\tau|^2I+\tau^2, \qquad \ast(\tau\wedge\tau)=-|\tau|^2\vp+\theta(\tau^2)\vp.
\end{equation}

\begin{definition}
A closed $G_2$-structure is said to be {\it quadratic} if
$$
d\tau = \frac{1}{7}(1+q)|\tau|^2\vp + q\ast(\tau\wedge\tau), \qquad \mbox{for some}\quad q\in\RR.
$$
\end{definition}

The notion was introduced in \cite[Remark 14]{Bry} as the most general way in which $d\tau$ can quadratically depend on $\tau$ (note that parallel structures are trivially quadratic for any $q$).  This is also evidenced in the first of the following two equivalent conditions to $\vp$ being quadratic, which can be deduced from \eqref{att} and \eqref{att2}:
\begin{equation}\label{quad-Q}
Q_{d\tau}=\tfrac{6q-1}{21}|\tau|^2I+q\tau^2;
\end{equation}
and
\begin{equation}\label{quad-Ric}
\Ricci=\tfrac{4q-3}{14}|\tau|^2I+(q-\unm)\tau^2.
\end{equation}

It is proved in \cite[Corollary 2]{Bry} that $g$ is an Einstein metric if and only if $\vp$ is quadratic for $q=\unm$, in which case $\ricci=-\frac{1}{14}|\tau|^2g$.  A quadratic $G_2$-structure with $q=\frac{1}{6}$, i.e.
$$
d\tau = \frac{1}{6}|\tau|^2\vp + \frac{1}{6}\ast(\tau\wedge\tau),
$$
is called {\it extremally Ricci-pinched} (ERP for short).  This concept was defined in \cite[Remark 13]{Bry} and owes its name to the following result.

\begin{theorem}\cite[Corollary 3]{Bry}\label{Bineq}
If $\vp$ is a closed $G_2$-structure on a compact manifold $M$, then
$$
\int_M R^2 \ast 1 \leq 3\int_M |\Ricci|^2 \ast 1,
$$
and equality holds if and only if $\vp$ is extremally Ricci pinched.
\end{theorem}

Note the factor of $3$ on the right hand side instead of the a priori factor of $7$ provided by the Cauchy-Schwartz inequality.  In particular, in the compact case, a closed $G_2$-structure is Einstein if and only if it is parallel (see also \cite{ClyIvn}); moreover, it is also proved in \cite{Bry} that for any ERP $G_2$-structure on a compact manifold, $\Ricci$ has only two eigenvalues: $-\frac{1}{6}|\tau|^2$ of multiplicity $3$ and $0$ of multiplicity $4$.

The following question naturally arises:

\begin{question}\label{ineq}
Does the inequality $\frac{R^2}{|\Ricci|^2}\leq 3$ hold for any closed $G_2$-structure in the homogeneous case?
\end{question}

Recall that this holds after integrating on $M$ in the compact case by Theorem \ref{Bineq}.  In Example \ref{Htype-ineq}, we give a negative answer to this question, and Example \ref{muAex} shows that there is no lower positive bound for  $\frac{R^2}{|\Ricci|^2}$ on homogeneous closed $G_2$-structures.

It is shown in \cite[4.66]{Bry} that the only possible value of $q$ for a quadratic (non-parallel) $G_2$-structure on a compact manifold is $q=\frac{1}{6}$ (see also \cite[Proposition 9.1]{Lty} for an alternative proof of the impossibility for $q=0$).  In the homogeneous case, there are in principle three values that are possible.

\begin{lemma}\label{qhom}
If $(M,\vp)$ is a homogeneous, closed and quadratic (non-parallel) $G_2$-structure for some $q$, then the only possibilities for $q$ are:
\begin{itemize}
\item $q=0$ and $\vp$ is an eigenform, with $\Delta\vp=\frac{1}{7}|\tau|^2\vp$;
\item[ ]
\item $q=\unm$ and $\vp$ is Einstein, with $\ricci=-\frac{1}{14}|\tau|^2g$;
\item[ ]
\item $q=\frac{1}{6}$, i.e.\ $\vp$ is extremally Ricci pinched.
\end{itemize}
\end{lemma}

\begin{proof}
Since in the homogeneous case the function $|\tau|^2$ is constant on $M$, it follows from \cite[4.69]{Bry} that
$$
7q(2q-1)\ast(\tau\wedge\tau\wedge\tau)=0.
$$
Thus either $q=0$, or $q=\unm$, or $\tau\wedge\tau\wedge\tau=0$.  The last alternative implies that either $q=\frac{1}{6}$ or $\vp$ is parallel by \cite[4.66]{Bry}, concluding the proof.
\end{proof}

It follows that for $q=0$ and $q=\unm$, the existence of homogeneous examples is equivalent to the following still open problems:

\begin{question}\label{qcero}
  Can a homogeneous closed (non-parallel) $G_2$-structure be an eigenform?
\end{question}

\begin{question}\label{Ehom}
  Is there any homogeneous closed (non-parallel) $G_2$-structure with Einstein associated metric?
\end{question}

A partial negative answer to Question \ref{Ehom} was given in \cite{FrnFinMnr2} in the case of left-invariant closed $G_2$-structures on solvable Lie groups.  This would be a definite answer if the Alekseevskii Conjecture, asserting that any homogeneous Einstein metric of negative scalar curvature is isometric to a solvmanifold, was known to be true.  Within the more general class of locally conformal closed $G_2$-structures, an example with Einstein associated metric was found on a solvable Lie group in \cite{FinRff1}.  On the other hand, some homogeneous closed $G_2$-structures do define Ricci soliton metrics (see \cite{FrnFinMnr,BF,LF,Ncl} and the examples below).

If one assumes that the answer to Question \ref{Ehom} is definitely no, then it is natural to wonder about what would be the extremally Ricci pinched $G_2$-structures in the homogeneous case:

\begin{question}\label{ineq2}
Is there a constant $C<7$ such that the inequality $\frac{R^2}{|\Ricci|^2}\leq C$ holds for any homogeneous closed $G_2$-structure?  What is the value of $C$ and its meaning?  Are the $G_2$-structures for which equality holds (if they exist at all) distinguished in some sense?
\end{question}

The last option $q=\frac{1}{6}$ in Lemma \ref{qhom} is indeed possible.  The following homogeneous example of an extremally Ricci pinched $G_2$-structure was given in \cite[Example 1]{Bry}.

\begin{example}\label{Bryant} ({\it Bryant's example of a homogeneous ERP $G_2$-structure})
Consider the homogeneous space $G/K=\Sl_2(\CC)\ltimes\CC^2/\SU(2)$ and its reductive decomposition
$$
\ggo=\slg_2(\CC)\ltimes\CC^2 = \sug(2)\oplus\pg, \qquad \pg=\herm(2)\oplus\CC^2.
$$
By taking the basis
$$
\left\{ Z_1=\left[\begin{smallmatrix} \im&0\\ 0&-\im \end{smallmatrix}\right],
Z_2=\left[\begin{smallmatrix} 0&-1\\ 1&0 \end{smallmatrix}\right],
Z_3=\left[\begin{smallmatrix} 0&\im\\ \im&0 \end{smallmatrix}\right]\right\},
$$
of $\sug(2)$, the basis of $\herm(2)$
$$
\left\{ e_1=\left[\begin{smallmatrix} -1&0\\ 0&1 \end{smallmatrix}\right],
e_2=\left[\begin{smallmatrix} 0&-\im\\ \im&0 \end{smallmatrix}\right],
e_3=\left[\begin{smallmatrix} 0&-1\\ -1&0 \end{smallmatrix}\right]\right\},
$$
and the basis
$$
\left\{ e_4=(1,0), e_5=(\im,0), e_6=(0,1), e_7=(0,-\im)\right\}
$$
of $\CC^2$, we obtain that the differential of $G$-invariant $1$-forms on $G/K$ satisfies that $de^1=de^2=de^3=0$ and
$$
\begin{array}{ll}
de^4=-e^{14}-e^{27}-e^{36}, &\qquad
de^5=-e^{15}-e^{26}+e^{37}, \\
de^6=e^{16}-e^{25}-e^{34}, &\qquad
de^7=e^{17}-e^{24}+e^{35}.
\end{array}
$$
It is easy to see that the $3$-form
$$
\vp=e^{123}+e^{145}+e^{167}+e^{246}-e^{257}-e^{347}-e^{356}\in\Lambda^3\pg^*,
$$
is $\Ad(K)$-invariant, by using that for any $Z\in\sug(2)$, $\ad{Z}|_{\herm(2)}=\ad_{\sug(2)}{Z}$ and $\ad{Z}|_{\CC^2}=Z$, so $\ad{Z}|_{\pg}\in\ggo_2$.  Thus $\vp$ defines a $G$-invariant $G_2$-structure on $G/K$.  A straightforward computation using the formula for $d$ above gives that
$$
d\vp=0, \quad \tau=6e^{45}-6e^{67}, \quad \Delta\vp=d\tau=12\vp-12e^{123}, \quad \ast(\tau\wedge\tau)=-72e^{123}.
$$
This implies that $d\tau=\frac{1}{6}|\tau|^2\vp+\frac{1}{6}\ast(\tau\wedge\tau)$, i.e.\ $(G/K,\vp)$ is extremally Ricci pinched.  Since $K$ is a maximal compact subgroup of $G$, the manifold $G/K$ is diffeomorphic to $\RR^7$.  However, by considering any torsion-free, cocompact and discrete subgroup of $\Sl_2(\CC)$ leaving invariant a lattice in $\CC^2$ (see \cite[Example 1]{Bry} for an explicit one) one obtains a discrete subgroup $\Gamma\subset G$ acting free and properly discontinuous on $G/K$ by automorphisms of $\vp$.  Thus the compact manifold $M=\Gamma\backslash G/K$ admits a $G_2$-structure, denoted also by $\vp$, such that $(M,\vp)$ is locally equivalent to $(G/K,\vp)$.   In particular, $(M,\vp)$ is also extremally Ricci pinched, though only locally homogeneous.
\end{example}

\section{Laplacian solitons}

The non-existence of shrinking or steady (non-parallel) Laplacian solitons was proved in \cite{Lin} in the compact case, though the following is still open:

\begin{question}\label{exp}
Are there compact and closed expanding Laplacian solitons?
\end{question}

Closed expanding Laplacian solitons have been found on solvable and nilpotent Lie groups in \cite{BF, LF, Ncl}, none of which is an eigenform.  However, the existence of the following classes of Laplacian solitons remained open:

\begin{question}\label{shr}
Are there closed shrinking Laplacian solitons?
\end{question}

\begin{question}\label{ste}
Are there closed steady Laplacian solitons (non-parallel)?
\end{question}

Homogeneous examples for these two questions will be given in Examples \ref{Bryant2} and \ref{Htype}.  In particular, the shrinking LSs in Example \ref{Htype} provide the first closed Laplacian flow solutions with a finite time singularity.  We note that the immortality of closed Laplacian flow solutions was proved on the class of left-invariant closed $G_2$-structures on solvable Lie groups with a codimension one abelian normal subgroup in \cite[Corollary 5.19]{LF}.  However, Example \ref{Htype} will show that this is not true for all solvmanifolds.

General properties of homogeneous Laplacian solitons were studied in \cite[Section 4]{LF}.  Let $G/K$ be a homogeneous space with $G$ simply connected, $K$ connected and compact, and consider the reductive decomposition $\ggo=\kg\oplus\pg$ such that $B(\kg,\pg)=0$, where $B$ is the Killing form of $\ggo$.

\begin{proposition}\cite[Proposition 4.5 and Remarks 4.6, 4.7]{LF}\label{hom-sol}
Let $\vp$ be a closed $G$-invariant $G_2$-structure on $G/K$ and consider the symmetric operator $Q_{d\tau}:\pg\longrightarrow\pg$ such that $\theta(Q_{d\tau})\vp=d\tau=\Delta\vp$.  Then, the following conditions are equivalent:
\begin{itemize}
\item[(i)] The $G$-invariant Laplacian flow solution starting at $\vp$ is given by $\vp(t)=c(t)f(t)^*\vp$ for some $c(t)\in\RR^*$ and $f(t)\in\Aut(G/K)$.
\item[ ]
\item[(ii)] $Q_{d\tau}=cI+\tfrac{1}{2}(D_\pg+D_\pg^t)$, for some $c\in\RR$, $D=\left[\begin{smallmatrix} 0&0\\ 0&D_\pg \end{smallmatrix}\right]\in\Der(\ggo)$.
\end{itemize}
In that case, $(G/K,\vp)$ is a Laplacian soliton with $\Delta\vp=-3c\vp-\lca_{X_D}\vp$ and the $G$-invariant Laplacian flow solution starting at $\vp$ is given by $\vp(t)=c(t)e^{s(t)D_\pg}\cdot\vp$, where $c(t):=(-2ct+1)^{3/2}$ and $s(t):=-\frac{1}{2c}\log(-2ct+1)$ (for $c=0$, set $s(t)=t$).
\end{proposition}

See Section \ref{homog} for the definitions of $\Aut(G/K)$ and $X_D$.  The homogeneous $G_2$-structure $(G/K,\vp)$ is called a {\it semi-algebraic soliton} when part (i) holds, and if in addition in part (ii) one has that $D^t\in\Der(\ggo)$, then $(G/K,\vp)$ is called an {\it algebraic soliton}.  It is worth pointing out that the concepts of semi-algebraic and algebraic solitons are applied to homogeneous spaces, not to homogeneous manifolds, since the conditions may depend on the presentation of $M$ as a homogeneous space $G/K$ chosen.  Algebraic solitons are geometrically characterized among homogeneous Laplacian solitons as those which are {\it Laplacian flow diagonal}: the family of symmetric operators $\{ Q_{d\tau(t)}:t\in(T_-,T_+)\}$ simultaneously diagonalizes (see \cite[Theorem 4.10]{LF}).

The above proposition can be used to show that Bryant's example is a steady Laplacian soliton as follows.  This provides a positive answer to Question \ref{ste}.

\begin{example}\label{Bryant2} ({\it Bryant's example is a steady LS})
It follows from \eqref{quad-Ric} that in Example \ref{Bryant},
$$
\Ricci=\Diag(-12,-12,-12,0,0,0,0), \quad \tau^2=\Diag(0,0,0,-36,-36,-36,-36),
$$
and thus
$$
Q_{d\tau}=\Diag(0,0,0,-6,-6,-6,-6) = D_\pg,
$$
for $D=\Diag(0,0,0,0,0,0,-6,-6,-6,-6)\in\Der(\ggo)$.  $(G/K,\vp)$ is therefore a steady LS (algebraic) by Proposition \ref{hom-sol}.  Note that this does not imply that the compact quotient $(\Gamma\backslash G/K,\vp)$ is a LS, as the field $X_D\in\mathfrak{X}(G/K)$ does not descend to $\Gamma\backslash G/K$.  The Laplacian flow evolution of $\vp$ on the compact manifold $\Gamma\backslash G/K$ is however eternal and somewhat striking, it may be described as `locally' self-similar.  Explicitly, $\vp(t)=e^{12t}\vp+(1-e^{12t})e^{123}$.

Since $\Ricci=-12I+E_\pg$ for the derivation $E$ of $\ggo$ defined by $E|_{\slg_2(\CC)}=0$ and $E|_{\CC^2}=12I$, we obtain that $(G/K,\ip)$ is an expanding Ricci soliton (see \cite[Section 3]{homRS}).
\end{example}

\begin{remark}
It is worth pointing out that the eternal Laplacian flow solution given in \cite[Example 2]{Bry} is not correct, one can check that the one-parameter family \cite[(6.24)]{Bry} does not solve the Laplacian flow equation.  The solution is actually given by
$$
\sigma(t) = \left(\left(\tfrac{10}{3}t+1\right)^{3/5} - 1\right) \omega^{123} + \sigma, \qquad t\in\left(-\tfrac{3}{10},\infty\right),
$$
and $\sigma$ is precisely the expanding LS worked out in \cite[Section 7]{BF}.  This follows from Proposition \ref{hom-sol} and has also been computed in \cite[Theorem 4.2]{FrnFinMnr}.
\end{remark}

We now study a $3$-parameter family of closed $G_2$-structures and find a new example with the same flavor.

\begin{example}\label{tripleA} ({\it New example of a homogeneous $G_2$-structure which is ERP and also a steady LS})
Let $\ggo$ be the solvable Lie algebra with basis $\{ e_1,\dots,e_7\}$ such that $\ngo:=\la e_3,e_4,e_5,e_6\ra$ is an abelian ideal and
$$
\ad{e_7}|_\ngo=\Diag(a,a,-a,-a), \quad
\ad{e_1}|_\ngo=\Diag(b,-b,b,-b), \quad
\ad{e_2}|_\ngo=\Diag(c,-c,-c,c).
$$
Equivalently, $de^1=de^2=de^7=0$ and
$$
\begin{array}{ll}
de^3=-ae^{73}-be^{13}-ce^{23}, &\quad de^4=-ae^{73}+be^{13}+ce^{23}, \\
de^5=ae^{73}-be^{13}+ce^{23}, &\quad de^6=ae^{73}+be^{13}-ce^{23}.
\end{array}
$$
Note that actually a family of Lie algebras depending on parameters $a,b,c\in\RR$ is being considered.  For the $G_2$-structure $\vp=\omega\wedge e^7+\rho^+$ defined in \eqref{phican}, a straightforward computation gives that $d\vp = 0$.  Thus $\vp$ is a closed $G_2$-structure on the solvable Lie group $G_{a,b,c}$ with Lie algebra $\ggo$, for any $a,b,c\in\RR$.  Up to isometry and scaling, we can assume in the nonflat case that
$$
a^2+b^2+c^2=3, \qquad a\geq b\geq c\geq 0.
$$
It is straightforward to see by using the formula for the differential $d$ above that the torsion $2$-form of $(G_{a,b,c},\vp)$ is given by
$$
\tau = -2ae^{34} -2be^{35} +2ce^{36} -2ce^{45} -2be^{46} +2ae^{56}, \qquad |\tau|^2=24,
$$
so as a matrix,
$$
\tau^2=-12\Diag\left(0,0,1,1,1,1,0\right).
$$
The Ricci operator and scalar curvature of $(G_{a,b,c},\ip)$ are respectively given by
$$
\Ricci=-4\Diag\left(b^2,c^2,0,0,0,0,a^2\right), \qquad R=-12,
$$
(see e.g.\ \cite[(25)]{solvsolitons}) so it follows from \eqref{ricop2} that
$$
Q_{d\tau} = \Diag(-4b^2+4, -4c^2+4, -2, -2, -2, -2, -4a^2+4).
$$
By using Proposition \ref{hom-sol} and the fact that any derivation of a solvable Lie algebra has its image contained in the nilradical, it is easy to prove that $(G_{a,b,c},\vp)$ is a semi-algebraic soliton, i.e. $Q_{d\tau}=kI+\unm(D+D^t)$ for some $k\in\RR$ and $D\in\Der(\ggo)$, if and only if either

\begin{itemize}
\item[(i)] $a=b=c=1$, $k=0$, so $(G_{1,1,1},\vp)$ is a steady Laplacian soliton (algebraic),
\item[(ii)] or $a=b=\sqrt{3/2}$, $c=0$ and $k=-2$,
\item[(iii)] or $a=\sqrt{3}$, $b=c=0$ and $k=-8$.
\end{itemize}

Note that in the last two cases $(G_{a,b,c},\vp)$ is an expanding Laplacian soliton (algebraic), and that each of the three isomorphism classes of Lie groups involved in the family does appear in the above list of solitons.  It follows from the formula for $\Ricci$ given above that the list also classifies those $(G,\ip)$ that are Ricci solitons, which are all expanding (see also \cite[Theorem 4.8]{solvsolitons}).  A straightforward computation gives that $\ast(\tau\wedge\tau) = -24e^{127}$ and
$$
\Delta\vp=d\tau = 4\left(a^2e^{347} + a^2e^{567} + b^2e^{135} - b^2e^{146} - c^2e^{236} - c^2e^{245}\right), \qquad\forall a,b,c.
$$
On the other hand, it follows from \eqref{quad-Q} that $(G,\vp)$ is quadratic if and only if it is extremally Ricci pinched, if and only if $a=b=c=1$.
Finally, consider the function from Question \ref{ineq}
$$
F(a,b,c)=\frac{R^2}{|\Ricci|^2} = \frac{(a^2+b^2+c^2)^2}{a^4+b^4+c^4},
$$
which provides an invariant up to equivalence and scaling for $G_2$-structures.  Note that $F\leq 3$, that is, the condition in Question \ref{ineq} holds among this family.  Moreover, $F(a,b,c)=3$ if and only if $a=b=c=1$.  We also have that $1\leq F$ with equality precisely at the expanding Laplacian soliton in part (iii) above.  The value of $F$ at the expanding Laplacian soliton in part (ii) is $2$.
\end{example}

\begin{remark}
It is difficult to see whether the above extremally Ricci pinched example is equivalent or not (up to scaling) to Bryant's example (see Example \ref{Bryant}), as the usual invariants all coincide.  This will be answered in the negative in Remark \ref{notequiv}.
\end{remark}

\begin{example}\label{triplA2} ({\it LF evolution of the family $(G_{a,b,c},\vp)$})
The bracket flow approach developed in \cite[Sections 3.2,3.3]{LF} can be used to study the Laplacian flow evolution of the family of $G_2$-structures considered in Example \ref{tripleA}.  Indeed, it is easy to see that the family of Lie brackets $\{\mu=\mu(a,b,c):a,b,c\geq 0\}$ on the vector space $\ggo$ defined at the beginning of the example is invariant by the bracket flow $\mu'=\delta_\mu(Q_\mu)$, which becomes the ODE for real functions $a=a(t)$, $b=b(t)$, $c=c(t)$ given by
$$
\left\{\begin{array}{l}
a'=(-4a^2+\tfrac{4}{3}(a^2+b^2+c^2))a, \\ \\
b'=(-4b^2+\tfrac{4}{3}(a^2+b^2+c^2))b, \\ \\
c'=(-4c^2+\tfrac{4}{3}(a^2+b^2+c^2))c.
\end{array}\right.
$$
We first note that this is precisely the gradient flow of the functional $H=\tfrac{1}{3}f-g$, where
$$
f(a,b,c):=(a^2+b^2+c^2)^2=\tfrac{1}{16}R^2, \qquad g(a,b,c):=a^4+b^4+c^4=\tfrac{1}{16}|\Ricci|^2.
$$
One has that $H\leq 0$, and equality holds if and only if $a=b=c$, i.e.\ $(G_{a,b,c},\vp)$ is a steady Laplacian soliton.  Since along the bracket flow $f'=8H\leq 0$, all the Laplacian flow solutions are immortal among this family.  By using that $g'=16\left(\frac{1}{3}f^{1/2}g-(a^6+b^6+c^6)\right)$, we obtain that our invariant functional $F=f/g$ satisfies for $f=1$ that
$$
g^2F'=16(-g^2+a^6+b^6+c^6)\geq 0,
$$
with equality only at the Laplacian solitons given in parts (i), (ii) and (iii) in Example \ref{tripleA}.  Thus $F$ is strictly increasing along any of these Laplacian flow solutions that is not a soliton.  The same behavior for $F$ is obtained for the bracket flow corresponding to the Ricci flow, which is minus the gradient flow of the functional $g$.  Concerning convergence, it is easy to see that for the bracket flow solution $\mu(t)$ starting at $\mu_0=\mu(a,b,c)$ we have that
$$
\mu/|\mu|\underset{t\to\infty}\longrightarrow\left\{\begin{array}{ll}
\tfrac{1}{\sqrt{24}}\mu(1,1,1), &\qquad a\geq b\geq c>0;\\ \\
\tfrac{1}{4}\mu(1,1,0), &\qquad a\geq b>0, \quad c=0;\\ \\
\tfrac{1}{\sqrt{8}}\mu(1,0,0), &\qquad a>0, \quad b=c=0.
\end{array}\right.
$$
It follows from \cite[Corollary 3.6]{LF} that on each of the Lie groups $G_{1,1,1}$, $G_{1,1,0}$ and $G_{1,0,0}$, the Laplacian flow solution $\vp(t)$ starting at any closed left-invariant $G_2$-structure satisfies that $r|\tau_t|_t^3\vp(t)$ smoothly converges up to pull-back by diffeomorphisms to the (corresponding) Laplacian soliton $\vp$, where $r$ equals $(24)^{-3/2}$, $(16)^{-3/2}$ and $(8)^{-3/2}$, respectively.
\end{example}

The following example is very generous.  On one hand, it provides the first examples of shrinking Laplacian solitons, answering Question \ref{shr}.  On a second hand, it shows that the functional $F=\frac{R^2}{|\Ricci|^2}$ is not always $\leq 3$ (see Example \ref{Htype-ineq} below).

\begin{example}\label{Htype} ({\it First examples of shrinking LSs})
We consider the family of solvable Lie algebras $\ggo$ depending on a parameter $a\in\RR$ with basis $\{ e_1,\dots,e_7\}$ such that $\hg:=\la e_1,\dots,e_6\ra$ is a $2$-step nilpotent ideal with Lie bracket
$$
[e_1,e_4]=-e_5, \quad [e_1,e_3]=-e_6, \quad [e_2,e_3]=-e_5, \quad [e_2,e_4]=e_6,
$$
and
$$
\ad{e_7}|_\hg=\Diag\left(a,a,\unm-a,\unm-a,\unm,\unm\right)\in\Der(\hg), \qquad a\in\RR.
$$
Endow the corresponding simply connected solvable Lie group $G_a$ with the left-invariant $G_2$-structure determined by the positive $3$-form $\vp=\omega\wedge e^7+\rho^+$ given in \eqref{phican}.  It is straightforward to see that $d\vp=0$ for any $a\in\RR$ and the torsion $2$-form of $(G_a,\vp)$ is given by
$$
\tau =2(1-a)e^{12}+(1+2a)e^{34}-3e^{56},
$$
so as a matrix,
$$
\tau^2=\Diag\left(-4(1-a)^2, -4(1-a)^2, -(1+2a)^2, -(1+2a)^2, -9,-9,0\right).
$$
It follows from formula \cite[(25)]{solvsolitons} that the Ricci operator of $(G_a,\ip)$ equals
$$
\Ricci=\Diag\left(-1-2a,-1-2a,-2+2a,-2+2a, 0,0,-1+2a-4a^2\right),
$$
and the remaining quantities we need to compute the operator $Q_{d\tau}$ using \eqref{ricop2} are the scalar curvature of $(G_a,\ip)$ and the norm of the torsion, which are respectively given by
$$
R=-7+2a-4a^2, \qquad |\tau|^2=14-4a+8a^2.
$$
We therefore obtain that
\begin{align}
Q_{d\tau} = \Diag( & -\tfrac{2}{3}+\tfrac{4}{3}a-\tfrac{2}{3}a^2, -\tfrac{2}{3}+\tfrac{4}{3}a-\tfrac{2}{3}a^2, -\tfrac{1}{6}-\tfrac{2}{3}a-\tfrac{2}{3}a^2, -\tfrac{1}{6}-\tfrac{2}{3}a-\tfrac{2}{3}a^2, \label{Qb0}\\
& -\tfrac{13}{6}-\tfrac{2}{3}a+\tfrac{4}{3}a^2, -\tfrac{13}{6}-\tfrac{2}{3}a+\tfrac{4}{3}a^2, \tfrac{4}{3}+\tfrac{4}{3}a-\tfrac{8}{3}a^2), \notag
\end{align}
which implies that the $G_2$-structures $(G_a,\vp)$ are all algebraic solitons (see Proposition \ref{hom-sol}).  Indeed, we have that
$$
Q_{d\tau}=cI+D, \qquad c=c(a)=\tfrac{4}{3}+\tfrac{4}{3}a-\tfrac{8}{3}a^2,
$$
where
\begin{align*}
D = \Diag( & -2+2a^2, -2+2a^2, -\tfrac{3}{2}-2a+2a^2, -\tfrac{3}{2}-2a+2a^2,  \\
& -\tfrac{7}{2}-2a+4a^2, -\tfrac{7}{2}-2a+4a^2, 0)\in\Der(\ggo).
\end{align*}
Since $(G_a,\vp)$ is equivalent to $(G_b,\vp)$ for $b=\unm-a$, we assume from now on that
$$
\unc\leq a<\infty.
$$
Note that if $a,b\in\left[\unc,\infty\right)$, then $G_a$ is isomorphic to $G_b$ if and only if $a=b$.  This implies that if $a\ne b$, then $(G_a,\vp)$ and $(G_b,\vp)$ can not be equivalent, not even up to scaling, since two completely solvable Lie groups endowed with left-invariant metrics which are isometric must be isomorphic (see \cite{Alk}).  This can also be shown by considering ratios of eigenvalues of $\Ricci$ in this particular case.

The number $c$ is strictly decreasing as a function of $a$ on $[\unc,\infty)$, giving rise to the three types of Laplacian solitons:
$$
\left\{\begin{array}{l}
c(a)>0, \quad \unc\leq a <1, \quad (G_a,\vp) \;\mbox{is a shrinking Laplacian soliton};  \\ \\
c(1)=0, \quad (G_1,\vp) \;\mbox{is a steady Laplacian soliton};\\ \\
c(a)<0, \quad 1<a<\infty, \quad (G_a,\vp) \;\mbox{is an expanding Laplacian soliton}.
\end{array}\right.
$$
On the other hand, $(G_a,\ip)$ is a Ricci soliton if and only if $a=1$ (see \cite[Theorem 4.8]{solvsolitons}).  We note that $(G_a,\ip)$ has $\Ricci\leq 0$ if and only if $\unc\leq a\leq 1$, though it never has nonpositive sectional curvature: the curvature of the plane generated by $e_3,e_5$ equals $\unm a>0$.  We give the following formulas for completeness:
$$
\Delta\vp = d\tau = -4a(1-a)e^{127} + (-1+4a^2)e^{347} + 3e^{567} +3\rho^+,
$$
$$
\ast(\tau\wedge\tau) = -6(1+2a)e^{127} - 12(1-a)e^{347} + 4(1-a)(1+2a)e^{567}.
$$
Recall that one can use the last paragraph in Proposition \ref{hom-sol} to explicitly obtain the Laplacian flow solution starting at each of the above $G_2$-structures.  For instance, the steady soliton solution (i.e.\ $a=1$) is given by $\vp(t)=e^{127}+e^{6t}(\vp-e^{127})$, and the shrinking Laplacian flow solutions have a finite-time singularity at $T=\frac{1}{2c}$ and are of the form
$$
\vp(t) = (-2ct+1)^{n_1}e^{127} + (-2ct+1)^{n_2}e^{347} + (-2ct+1)^{n_3}e^{567} + (-2ct+1)^{n_4}\rho^+,
$$
for some numbers $n_1\geq n_2>n_3>n_4>0$.  Thus $\vp(t)\to 0$ as $t\to T$.  The asymptotic behavior of the soliton is determined by the largest exponent $n_1=\frac{3a}{2(1+2a)}$, which is strictly increasing on $a$ and satisfies $\unc\leq n_1<\unm$ for $\unc\leq a<1$.
\end{example}

\begin{example}\label{Htype-ineq} ({\it The value of $\frac{R^2}{|\Ricci|^2}$ can be $>3$})
In Example \ref{Htype}, the invariant functional
$$
F(a)=\frac{R^2}{|\Ricci|^2} = \frac{(-4a^2+2a-7)^2}{16a^4-16a^3+28a^2-12a+11},
$$
gives an alternative proof of the fact that the family $(G_a,\vp)$ is pairwise non-homothetic, as $F$ strictly decreases on $\left[\unc,\infty\right)$.  Moreover, $F\geq 1$,
$$
F(\unc)=\tfrac{81}{17}>3=F(1), \qquad \lim_{a\to\infty} F(a) = 1 \qquad \left(\tfrac{81}{17}\approx 4.76\right).
$$
Thus the shrinking Laplacian solitons $(G_a,\vp)$, $\unc\leq a<1$ all provide counterexamples to Question \ref{ineq}.  It follows from \eqref{quad-Q} and the formulas for $Q_{d\tau}$ and $\tau^2$ given in Example \ref{Htype} that $(G_a,\vp)$ is quadratic if and only if it is ERP, if and only if $a=1$.  Thus the steady LS $(G_1,\vp)$ is also ERP.
\end{example}

The following natural questions arise from what we have seen in all examples above.

\begin{question}\label{ERP-ste}
Is there any kind of interplay between extremally Ricci pinched $G_2$-structures and steady Laplacian solitons?
\end{question}

We now show that the ERP $G_2$-structure $(G_1,\vp)$ in the above example is equivalent to Bryant's example.

\begin{example}\label{Bryant3} ({\it $(G_1,\vp)$ is equivalent to Bryant's example})
Back to Example \ref{Bryant}, if we set
$$
f_1:=e_1=\left[\begin{smallmatrix} -1&0\\ 0&1 \end{smallmatrix}\right], \quad f_2:=\unm(-e_2+Z_3)=\left[\begin{smallmatrix} 0&\im\\ 0&0 \end{smallmatrix}\right], \quad f_3:=-\unm(e_3+Z_2)=\left[\begin{smallmatrix} 0&1\\ 0&0 \end{smallmatrix}\right],
$$
then $\ggo=\sug(2)\oplus\sg$, where $\sg:=\la f_1,f_2,f_3,e_4,e_5,e_6,e_7\ra$ is a solvable Lie subalgebra of $\ggo$.  It follows that the corresponding solvable Lie subgroup $S$ meets $K$ only at the identity element and thus $S$ acts simply and transitively on $G/K$.  In this way, $\vp$ defines a left-invariant $G_2$-structure $\psi$ on $S$ such that $(S,\psi)$ and $(G/K,\vp)$ are equivalent.  The nilradical of $\sg$ is $\ngo=\la f_2,\dots,e_7\ra$ and its Lie bracket is given by
$$
[f_2,e_6]=e_5, \quad [f_2,e_7]=e_4, \quad [f_3,e_6]=e_4, \quad [f_3,e_7]=-e_5,
$$
so $\ngo$ is isomorphic to the Lie algebra $\hg$ in Example \ref{Htype}.  We furthermore have that
$$
\ad_\sg{f_1}|_{\ngo}=\Diag(-2,-2,-1,-1,1,1),
$$
which implies that $\sg$ is isomorphic to the Lie algebra in Example \ref{Htype} for $a=1$.  It is easy to check that $(S,\psi)$, and consequently $(G/K,\vp)$, is indeed equivalent to the steady soliton $(G_1,\frac{1}{8}\vp)$.
\end{example}

\begin{remark}
In the above example, in spite of the solvable Lie group $G_1$ does not admit a lattice, as it is not even unimodular, there exists a torsion-free and discrete subgroup $\Gamma_1\subset\Aut(G_1,\vp)$ such that $M_1:=\Gamma_1\backslash G_1$ is a compact manifold and $(M_1,\vp)$ is locally equivalent to $(G_1,\vp)$ (see the last paragraph in Example \ref{Bryant}).  The moral of this is that one may not need a lattice in the Lie group $G$ itself to get a locally homogeneous compact version of a left-invariant geometric structure on $G$.
\end{remark}

\begin{remark}\label{notequiv}
Another consequence of Example \ref{Bryant3} is that Bryant's example $(G/K,\vp)$ is not equivalent (even up to scaling) to the extremally Ricci pinched $G_2$-structure given in Example \ref{tripleA} for $a=b=c=1$.  Indeed, such solvable Lie group is not isomorphic to $G_1$, but it is completely solvable as $G_1$, so they can not admit isometric left-invariant metrics (see \cite{Alk}).
\end{remark}

It is worth noticing that on the families given in both Examples \ref{tripleA} and \ref{Htype} the functional $F=\frac{R^2}{|\Ricci|^2}$ satisfies that $F\geq 1$, suggesting that there could be a lower bound for $F$ on homogeneous closed $G_2$-structures. The following example shows that this is not the case.

\begin{example}\label{muAex} ({\it No positive lower bound for $\frac{R^2}{|\Ricci|^2}$})
In \cite[Section 5]{LF}, almost abelian solvable Lie groups endowed with left-invariant closed $G_2$-structures are studied. Using the notation in that paper, we consider the family $(G_A,\vp)$ for the matrices
$$
A=\left[\begin{array}{c|c}
B&0\\\hline
0&B\\
\end{array}\right]\in\slg(3,\CC), \qquad B=\left[\begin{smallmatrix}
a&-1&0\\
1&-a&0\\
0&0&0
\end{smallmatrix}\right], \qquad a\in\RR.
$$
It follows from the formula for $\Ricci$ in \cite[(26)]{LF} that the functional $F=\frac{R^2}{|\Ricci|^2}$ on this family is given by
$$
F(a) = \frac{a^4}{a^4+2a^2},
$$
and so $F(a)\to 0$, as $a\to 0$.  Note that at $a=0$, one obtains a parallel $G_2$-structure which is equivalent to the flat $\RR^7$ by \cite[Proposition 5.6]{LF}.
\end{example}

\begin{example}\label{muAex2} ({\it Divergence of $\frac{R^2}{|\Ricci|^2}$ at the parallel structure})
If in the above example we consider instead
$$
A=\left[\begin{array}{c|c}
B&0\\\hline
0&B\\
\end{array}\right]\in\slg(3,\CC), \qquad B=\left[\begin{smallmatrix}
b&-1&0\\
1&b&0\\
0&0&-2b
\end{smallmatrix}\right], \qquad b>0,
$$
then we obtain a continuous family of expanding LSs (see \cite[Proposition 5.22]{LF}).  Let $A_0$ denote the parallel $G_2$-structure equivalent to the flat $\RR^7$ obtained for either $a=0$ in the above example or $b=0$ in this example.  Using \cite[(26)]{LF}, it is easy to see that on this new family, $F(b)\equiv 1$.  This implies that the limit of $F$ as $A$ goes to $A_0$ does not exist.
\end{example}

\end{document}